\documentclass[12pt,reqno]{amsart}
\usepackage{amsmath,amssymb,amsfonts,amsthm,amscd,amstext,amsxtra,amsopn,array,url,verbatim,mathrsfs}
\usepackage{graphicx}
\usepackage{amsmath,amssymb,amsfonts,amsthm,amssymb,amscd,url,amstext,amsxtra,amsopn
,txfonts}
\usepackage{verbatim}
\usepackage{fullpage}
\usepackage[dvipsnames,usenames]{xcolor}
\usepackage[colorlinks=true,urlcolor=black,citecolor=black,linkcolor=black]{hyperref}

\makeatletter
\@namedef{subjclassname@2010}{%
  \textup{2010} Mathematics Subject Classification}
\makeatother

\usepackage{url}
\usepackage{amsthm}
\usepackage{amsmath}
\usepackage{amssymb}

\newtheorem{theorem}{Theorem}[section]
\newtheorem{Lem}[theorem]{Lemma}

\newtheorem{corollary}[theorem]{Corollary}
\newtheorem{hypothesis}[theorem]{Hypothesis}

\newtheorem{remarks}[theorem]{Remarks}
\newtheorem{notation}[theorem]{Notation}
\theoremstyle{definition}

\numberwithin{equation}{section}

\begin{document}

\title[A log-free zero-density estimate for  $L$-functions]{A log-free zero-density estimate and small gaps in coefficients of L-functions}

\author{Amir Akbary}
\address{University of Lethbridge, Department of Mathematics and Computer Science, 4401 University Drive, Lethbridge, AB, T1K 3M4, Canada}
\email{amir.akbary@uleth.ca}
\author{Timothy S. Trudgian}
\address{The Australian National University, Mathematical Sciences Institute, 
Canberra, ACT 0200, Australia}
\email{timothy.trudgian@anu.edu.au}


\thanks{Research of the authors is partially supported by NSERC. The second author research is also partially supported by ARC}



\date{\today}

\begin{abstract}
Let $L(s, \piup\times\piup^\prime)$ be the Rankin--Selberg $L$-function attached to automorphic representations $\piup$ and $\piup^\prime$. Let $\tilde{\piup}$ and $\tilde{\piup}^\prime$ denote the contragredient representations associated to $\piup$ and $\piup^\prime$. Under the assumption of certain upper bounds for  coefficients of the logarithmic derivatives of $L(s, \piup\times\tilde{\piup})$ and   $L(s, \piup^\prime\times\tilde{\piup}^\prime)$, we prove a log-free zero-density estimate for $L(s, \piup\times\piup^\prime)$ which generalises a result due to Fogels in the context of Dirichlet $L$-functions. We then employ this log-free estimate in studying the distribution of the Fourier coefficients of an automorphic representation $\piup$. As an application we examine the non-lacunarity of the Fourier coefficients $b_f(p)$ of a modular newform $f(z)=\sum_{n=1}^{\infty} b_f(n) e^{{2\pi i n z}}$ of weight $k$, level $N$, and character $\chi$. More precisely for $f(z)$ and a prime $p$, set $j_f(p):=\max_{x;~x> p} J_{f} (p, x)$, where
$J_{f} (p, x):=\#\{{\rm prime}~q;~a_{\piup}(q)=0~{\rm for~all~}p<q\leq x\}.$
We prove that $j_f(p)\ll_{f, \theta} p^\theta$ for some $0<\theta<1$. 
\end{abstract}

\maketitle

\section{Introduction and Results}\label{into}
In the absence of powerful zero-free regions for $L$-functions it is worthwhile to study the number of zeroes in rectangles in the complex plane. To this end, for an $L$-function $L(s)$,  one considers the function
\begin{equation*}\label{Nsigdef}
N_{L}(\sigma, T) = \# \left\{\rho = \beta + i\gamma; \, L(\rho) = 0, \, \beta \geq \sigma, \, |\gamma|\leq T\right\}.
\end{equation*}
For a general $L$-function satisfying certain properties it is possible to prove, for $\epsilon>0$, the existence of a positive constant $c$ 
\begin{equation}\label{KPdensity}
N_L(\sigma, T) \ll_{\epsilon} T^{c(1-\sigma) + \epsilon}
\end{equation}
uniformly for $0\leq \sigma \leq 1$ (see \cite[Lemma 3]{KacPer} for details on such a theorem). It is expected that $L(1+it)\neq 0$ for a general $L$-function, which makes the estimation given in Theorem \eqref{KPdensity} trivial when $\sigma\rightarrow 1^{-}$. For many $L$-functions we can replace $T^\epsilon$ in \eqref{KPdensity} by a power of $\log{T}$. It would be desirable to sharpen \eqref{KPdensity} near the line $\sigma=1$ by removing the constant $\epsilon$ in the exponent of $T$. We call such a bound a \emph{log-free zero density estimate}. In this paper we achieve this for certain automorphic $L$-functions. 
In order to proceed  we need to introduce some notation and terminology.

Let $\piup$ be an automorphic cuspidal representation of ${\rm GL}_m(\mathbb{A}_\mathbb{Q})$  with unitary central character (for simplicity we call such $\piup$ an {\em automorphic representation}) and let $L(s, \piup)$ be its associated $L$-function,  which is written as a product of the local $L$-functions $L(s, \piup_p)$. Hence $L(s, \piup) = \prod_{p < \infty} L(s, \piup_{p})$, where, for $\Re(s)>1$, 
$$L(s, \piup_{p}) = \prod_{j=1}^{m} \left(1- \frac{\alpha_{\piup}(j, p)}{p^{s}}\right)^{-1}.$$
The complex numbers $\alpha_\piup(j, p)$, for $j=1, \ldots, m$, are the {\em local parameters at $p$}, where $m$ is the {\em degree} of $L(s, \piup)$. Associated to $\piup$ there is an integer $q_\piup\geq 1$, called the {\em conductor} of $\piup$, such that $\alpha_\piup(j, p)\neq 0$ for $p\nmid q_\piup$ and $1\leq j \leq m$.   The {\em generalised Ramanujan conjecture} (GRC) for $\piup$ (or $L(s, \piup)$)  is the assertion that for all $p\nmid q_\piup$ we have $|\alpha_\piup(j, p)|= 1$, for $j=1, \ldots, m$, and $|\alpha_\piup(j, p)|<1$ otherwise. Any $L(s, \piup)$ is called a {\em principal} $L$-function of ${\rm GL}_m(\mathbb{A}_\mathbb{Q})$. 
It is conjectured that GRC holds for all principal $L$-functions. 

Associated to $\piup$ is its {\em contragredient}, which itself is an automorphic representation. The set of local parameters for $\tilde{\piup}$ coincides with the set of complex conjugates of local parameters for ${\piup}$. So $\{\alpha_{\tilde{\piup}}(j, p)\}=\{\overline{\alpha_\piup(j, p)}\}.$ A representation $\piup$ is called {\em self-dual} if $\{\alpha_{{\tilde{\piup}}}(j, p)\}=\{\alpha_{{\piup}}(j, p)\} $. For a  self-dual $\piup$ the coefficients of the Dirichlet series representation of $L(s, \piup)$ are real.

We set $$ a_{\piup}(p^\ell)=\sum_{j=1}^m \alpha_{\piup}(j, p)^{\ell}.$$
If $n$ is not a perfect square we define $a_\piup(n)=0$. We postulate the following hypothesis regarding the average values of $\Lambda(n)|a_{\piup}(n)|^2$ over short intervals, where $\Lambda(n)$ is the von Mangoldt function.
\begin{hypothesis}
\label{HypothesisA}
There is an $\epsilon_\piup>0$ such that
$$\sum_{X<n\leq X+Y} \Lambda(n) |a_\piup(n)|^2 \ll_\piup Y,$$
whenever $Y\geq X^{1-\epsilon_\piup}$.
\end{hypothesis}
The above hypothesis states that the average size of $\Lambda(n)|a_{\piup}(n)|^2$ stays bounded over intervals of length $Y\geq X^{1-\epsilon_\piup}$. Hypothesis \ref{HypothesisA} is expected to be true for all automorphic representations $\piup$. In fact, under the assumption of GRC we have $$\sum_{X<n\leq X+Y} \Lambda(n) |a_\pi(n)|^2 \leq m^2 \left( \psi(X+Y)-\psi(X)\right),$$
where $\psi(X)$ is the classical Chebyshev function. From the Brun--Titchmarch inequality (see \cite[Theorem 6.6]{IK}) and the prime number theorem we can conclude that $$\psi(X+Y)-\psi(X)\ll Y$$
for $Y\geq X^{\theta}$  with $\theta>1/2$, which establishes Hypothesis \ref{HypothesisA} for any $0<\epsilon_\piup<1/2$.

The \emph{Rankin--Selberg $L$-function} $L(s, \piup \times \piup^\prime)$  associated to automorphic representations $\piup$ and $\piup^\prime$ of ${\rm GL}_m(\mathbb{A}_\mathbb{Q})$ and ${\rm GL}_{m^\prime}(\mathbb{A}_\mathbb{Q})$ is given by the local factors at primes $p$. We have
$$L(s, \piup\times\piup^\prime)=\prod_p L(s, \piup_p \times \piup_p^\prime),$$ 
where, for $\Re(s)>1$,
$$L(s, \piup_{p}\times \piup_{p}') = \prod_{\substack{j, k\\ 1 \leq j \leq m \\1 \leq k \leq m'}} \left(1- \frac{\alpha_{\piup\times\piup'}(j,k, p)}{p^{s}}\right)^{-1}.$$   
If $p\nmid (q_\piup, q_{\piup'})$ we have $\alpha_{\piup\times\piup'}(j,k, p)=\alpha_{\piup}(j, p) \alpha_{\piup'}(k, p),$
where $\alpha_{\piup}(p, j) $ and $\alpha_{\piup^\prime}(p, j) $ denote the local parameters associated to $\piup$ and $\piup^\prime$ (see \cite[p.\ 1460]{B}). We also set
$$ a_{\piup\times\piup'}(p^\ell)=\sum_{j=1}^m \sum_{k=1}^{m'}\alpha_{\piup\times\piup^\prime}(j, k, p)^{\ell}$$
and note that 
$$ a_{\piup\times\piup'}(p^\ell)= a_{\piup}(p^\ell) a_{\piup'}(p^\ell),$$
if $p\nmid (q_\piup, q_{\piup'})$.
The archimedean local factor $L(s,\piup_\infty \times \piup_\infty^\prime)$ is defined by
$$L(s,\piup_\infty \times \piup_\infty^\prime)=\prod_{\substack{j, k\\ 1 \leq j \leq m \\1 \leq k \leq m'}} \Gamma_\mathbb{R}(s+\mu_{\piup\times\piup^\prime} (j, k)),$$
where $\Gamma_\mathbb{R}(s)=\pi^{-s/2} \Gamma(s/2)$ and $\mu_{\piup\times\piup^\prime} (j, k)$ are complex numbers.
We suppose that the central characters of $\piup$ and $\piup^\prime$ are trivial on the product of positive reals when embedded diagonally into the archimedean places of the ideles. Then it is known that the completed $L$-function $$\Phi(s, \piup\times \piup^\prime)=L(s, \piup_\infty \times \piup_\infty^\prime )L(s,\piup\times\piup^\prime)$$ is entire unless $\piup^\prime=\tilde{\piup}$ in which case it has simple poles at $s=0$ and $s=1$. The completed $L$-function $\Phi(s, \piup)$ satisfies the functional equation
\begin{equation}\label{funceg}
\Phi(s, \piup\times\piup^\prime) = \tau(\piup\times \piup^\prime)q_{\piup\times\piup^\prime}^{-s} \Phi(1-s, \tilde{\piup}\times {\tilde{\piup}}^\prime),
\end{equation}
where the integer $q_{\piup\times\piup^\prime}>0$ is the \emph{conductor} of $\piup\times\piup^\prime$, and $\tau(\piup\times \piup^\prime)$, the so-called \emph{root number}, is a complex number of absolute value $1$.
We also define the conductor of $\Phi(s, \piup\times\piup^\prime)$ to be $$Q_{\piup\times\piup^\prime}=q_{\piup\times\piup^\prime}\prod_{\substack{j, k\\ 1 \leq j \leq m \\1 \leq k \leq m'}} (|\mu (j, k)|+2).$$
For more information on automorphic $L$-functions see \cite[Section 5.12]{IK} and references therein.

We set
\begin{equation*}
N_{\piup\times\piup^\prime}(\sigma, T) = \# \left\{\rho = \beta + i\gamma; \, L(\rho, \piup\times\piup^\prime) = 0, \, \beta \geq \sigma, \, |\gamma|\leq T\right\}.
\end{equation*}
Our main result is the following.
\begin{theorem}\label{thm3}
Assume that Hypothesis \ref{HypothesisA} holds for $\piup$ and $\piup^\prime$ and that either $m, m^\prime\leq 2$ or at least one of $\piup$ and $\piup^\prime$ is self-dual. Then there exist positive constants $c$ and $T_0$ (depending on $\piup$ and $\piup'$) such that for any $T\geq T_0\geq \max\{3, Q_{\piup\times\piup^\prime}\}$ we have
\begin{equation*}
N_{\piup\times\piup^\prime}(\sigma, T) \leq T^{c(1-\sigma)}
\end{equation*}
uniformly for $0\leq \sigma \leq 1$.
\end{theorem}
Theorem \ref{thm3} generalises Fogels'  log-free zero-density theorem \cite{Fogels} for Dirichlet $L$-functions.  A log-free zero-density estimate for Dirichlet $L$-functions (in a range different from Fogels') was first  developed by Linnik in the proof of his celebrated theorem on the least prime in an arithmetic progression and strengthened later by Selberg and others.  See \cite{J} for later developments in log-free estimates in the classical setting. There are also analogues of Theorem \ref{thm3} for automorphic $L$-functions in the level aspect due to Kowalski and Michel \cite{KM}.

\begin{remarks}
(i) In \cite[Theorem 2.3]{LY2} it is proved that for a self-dual $\piup$ there exists a constant $c>0$ such that
\begin{equation*}
\label{PNTLY}
\sum_{n\leq X} \Lambda(n) |a_\piup(n)|^2 =X+O\left(X\exp(-c\sqrt{\log{X}})\right).
\end{equation*}
This implies the existence of a positive constant $c^\prime$ such that 
\begin{equation}\label{Lep}
\sum_{X<n\leq X+Y} \Lambda(n) |a_\piup(n)|^2 \sim Y,
\end{equation}
for $Y\geq X\exp(-c^\prime\sqrt{\log{X}})$ as $X\rightarrow \infty$. Note that this does not imply Hypothesis \ref{HypothesisA}. It does not appear that one can remove the condition in Theorem \ref{thm3} by employing (\ref{Lep}) over the longer range $Y\geq X\exp(-c^\prime\sqrt{\log{X}})$.

(ii) The self-dual condition in Theorem \ref{thm3} is necessary to ensure the existence of a standard zero-free region (see Lemma \ref{zerofree}). 
\end{remarks}
We describe some applications of Theorem \ref{thm3} to the distribution of the values ${a_\piup(p)=\sum_{j=1}^{m} \alpha_j(p)}$ in short intervals. These applications resemble the prime number theorem over short intervals which was proved by Hoheisel in 1930. Hoheisel showed that there exists a constant $\theta\leq \frac{32999}{33000}$ for which 
$$\pi(X+X^\theta)-\pi(X) \sim \frac{X^\theta}{\log{X}},$$
as $X\rightarrow \infty$, where $\pi(X)$ denotes the prime counting function. {The bound on $\theta$ has been improved by several authors; the best bound, due to Huxley \cite{Hux}, namely $\theta \leq {7}/{12}+\epsilon$.}
Hoheisel's proof was accomplished by employing the explicit formula for the Riemann zeta-function $\zeta(s)$, a \emph{non-standard zero-free region} in the form  $\sigma \geq 1-A{\log\log t}/{\log t}$, first found by Littlewood (see \cite[Theorem 5.17]{Titchmarsh}), 
and a zero-density estimate involving log powers (see \cite[Theorem 9.18]{Titchmarsh}).

For a general $L$-function $L(s)$ for which $\log{L(s)}$ has a Dirichlet series representation for $\Re(s)>1$,  we can define, analogous to the Chebyshev function $\psi(X)$, 
$$\psi_L(X)=\sum_{n\leq X} \Lambda_L(n),$$
where  $-L^\prime(s)/L(s)=\sum_{n=1}^{\infty} \Lambda_L(n)/n^s$ for $\Re(s)>1$. 
Moreover we say that $L(s)$ has a {\em standard zero-free region} if we can find positive constants $A$ and $t_0$ such that 
\begin{equation*}
\label{zerofree1}
L(s)\neq 0,~{\rm whenever}~\sigma>1-\frac{A}{\log{t}}~ {\rm and} ~t>t_0.
\end{equation*}
%

%

In \cite{Moreno} Moreno proved that for a Dirichlet series $L(s)$ with an appropriate explicit formula and a \emph{standard} zero-free region one could establish a lower bound for $\psi_L(X+Y)-\psi_L(X)$, for $Y\geq X^{1-\epsilon_L}$, provided that $L(s)$ satisfies a \emph{log-free} zero-density estimate. Inspired by Moreno's observation we established the log-free estimate of Theorem  \ref{thm3} which is applicable for a large class of automorphic $L$-functions. Thus we can apply Moreno's theorem to these $L$-functions. 
By employing properties of Rankin--Selberg $L$-functions and the Tauberian theorem of Wiener--Ikehara we can prove 
\begin{equation}
\label{PNT}
\psi_{\piup\times\tilde{\piup}}(X)=\sum_{n\leq X} \Lambda(n) a_{\piup\times\tilde{\piup}}(n) \sim X,
\end{equation}
as $X\rightarrow \infty$. This can be considered as the prime number theorem for $L(s, \piup\times\tilde{\piup})$. (Here $\piup$ is not necessarily self-dual.) As a consequence of our log-free zero-density estimate we prove the following short-interval version of \eqref{PNT}.

\begin{theorem}
\label{thm4}
Suppose that the self-dual automorphic representation $\piup$ satisfies Hypothesis \ref{HypothesisA}. Then there is a $\nu_\piup$ with $0<\nu_\piup<1$ such that for all $\theta>\nu_\pi$,
\begin{equation*}
\label{HypothesisA2}
\psi_{\piup\times\tilde{\piup}}(X+Y)-\psi_{\piup\times\tilde{\piup}}(X) \asymp_\piup Y,
\end{equation*}
whenever $Y\geq X^\theta$.
\end{theorem}
The following result is a direct consequence of Theorem \ref{thm4} and GRC.

\begin{corollary}
\label{cor1.7}
Under the assumption of GRC for a self-dual $\piup$, there is a  $\nu_\piup$ with $0<\nu_\piup<1$ such that for all $\theta>\nu_\piup$,
\begin{equation*}
\sum_{X< n\leq X+Y} \Lambda(n) |a_\piup(n)|^2 \asymp_\piup Y,
\end{equation*}
and 
\begin{equation*}
\sum_{X< p\leq X+Y} (\log{p}) |a_\piup(p)|^2 \asymp_\piup Y,
\end{equation*}
whenever $Y\geq X^\theta$.
\end{corollary}

We next describe an application of Theorem \ref{thm4} in studying the non-lacunarity of the sequence $(a_\piup(p))$ where $p$ ranges over primes.  This problem has classical roots. Let $L(s, \Delta)=\sum_{n=1}^{\infty} \frac{\tau(n)}{n^s}$ be the Dirichlet series associated to the discriminant function $$\Delta(z)=e^{2\pi iz} \prod_{n=1}^{\infty} (1-e^{2\pi i nz})^{24}=\sum_{n=1}^{\infty} \tau(n) e^{2\pi i nz},$$
where $z$ is chosen in the upper half-plane. The coefficient $\tau(n)$ is the Ramanujan $\tau$-function. We know that $\Delta(z)$ is a cusp form of weight $12$ and level $1$, whence $L(s, \Delta)$ is an automorphic $L$-function for ${\rm GL}_2(\mathbb{A}_\mathbb{Q})$. More specifically $a_{\piup_\Delta}(p)=\tau(p)$, where $\piup_\Delta$ is the automorphic representation associated to $\Delta(z)$. A celebrated conjecture of Lehmer states that $\tau(n)\neq 0$ for all positive integers $n$. Serre \cite[Corollary 2, p.\ 174]{Serre2} has shown that for $x\geq 2$ and for any $\eta<1/2$, we have
\begin{equation*}\label{serr}
\#\{p\leq x: \tau(p) = 0\} \ll \frac{x}{\log^{1+ \eta} x}.
\end{equation*}
Since $\tau(n)$ is multiplicative, as a consequence of this result, Serre \cite[p.\ 179, Example 2]{Serre} proved that 
$\tau(n)$ is non-zero on a set of positive density. A sequence $(a(n))$ is said to be {\em non-lacunary} if the set $\{n;~a(n)\neq 0\}$ has positive density. More generally, Serre \cite[Theorem 16]{Serre}  proved  that the set of the  Fourier coefficients of a non-CM modular form of weight $k\geq 2$, level $N$, and character $\chi$, that is a normalised eigenform for the Hecke operators, is non-lacunary.

For the coefficients $b_f(n)$ of a modular form $f\neq 0$ of weight $k$, level $N$, and character $\chi$ one can also study the non-lacunarity of $(b_f(n))_{n\geq 1}$ 
by considering the function
\begin{equation*}
i_f(n):=\max\{i\geq 0;~b_f(n+k)=0~{\rm for~all}~0< k\leq i\}.
\end{equation*}
In \cite[p.\ 183]{Serre} Serre proposed this function and proved that if $f(z)$ is a cusp form of weight $k\geq 2$ that is not a linear combination of forms with complex multiplication, then 
\begin{equation}
\label{if}
i_f(n) \ll n.
\end{equation} 
By employing the asymptotic $$\sum_{n\leq X} |b_f(n)|^2 n^{1-k}=C_f X+O(X^{3/5}),$$ which is a classical result of Rankin and Selberg, one can improve \eqref{if} to
$i_f(n)\ll n^{3/5}.$
In \cite{BO}, Balog and Ono studied the non-vanishing of $b_f(n)$ over short intervals, and, consequently,  they proved 
$i_f(n) \ll_{f, \theta} n^\theta$ for $\theta>17/41$.
The exponent $17/41$ in this theorem was improved to $7/17$ in \cite[Corollary 1, p.\ 299]{KRW}.

We can also consider analogous problems for the sequence $(b_f(p))$ where $p$ ranges over primes. Here we study this problem in the context of automorphic $L$-functions. 

For a principal $L$-function $L(s, \piup)$, a prime variable $p$ and a real variable $x$ with $x> p$, let
$$J_{\piup} (p, x):=\#\{{\rm prime}~q;~a_{\piup}(q)=0~{\rm for~all~}p<q\leq x\}$$
and define
$$j_\piup (p)=\max_{x;~x> p} J_\piup (p,x)\}.$$

We are interested in finding a non-trivial upper bound for  $j_\piup(p)$.
%
Under the assumption of GRC for a self-dual $\piup$ we can deduce from \eqref{Lep} that
$$\#\{p;~X<p\leq X+Y~{\rm and}~a_\piup(p)\neq 0\} \gg \frac{Y}{\log{X}},$$
provided that $Y\geq X\exp(-c^\prime \sqrt{\log{x}})$ for a suitable $c^\prime>0$. This implies that
\begin{equation}
\label{PN}
j_\piup(p)\ll_\piup p \exp(-c^\prime \sqrt{\log{p}}).
\end{equation}
We prove the following refinement of \eqref{PN}.
\begin{theorem}
\label{thm1.11}
Under the assumption of GRC for a self-dual $\piup$, there is a $\nu_\piup$ with $0<\nu_\piup<1$ such that for $\theta>\nu_\piup$ and $Y\geq X^\theta$ we have
$$\#\{p;~X<p\leq X+Y~{\rm and}~a_\piup(p)\neq 0\} \gg_{\piup, \theta} \frac{Y}{\log{X}}.$$ 
Moreover, we have $$j_\piup(p)\ll_{\piup, \theta} p^\theta.$$
\end{theorem} 
Observe that in the case in which $\piup$ is associated to a cusp newform $f(z)$ then $a_\piup(p)=b_f(p)$ (the $p$-th Fourier coefficient of $f(z)$). It is known that if $f(z)$ is of CM-type then the set of primes $p$ with $b_f(p)\neq 0$ has density $\delta=1/2$ (\cite[p.\ 253]{MMS}). If $f(z)$ is not of CM-type then a conjecture of Lang and Trotter (\cite[p.\ 253]{MMS})  predicts that $\delta=1$. In fact for non-CM forms of weight $\geq 4$, Atkin and Serre \cite[p.\ 244]{Serre2} conjectured that 
$$|b_f(p)|\gg_\epsilon p^{\frac{k-3}{2}-\epsilon}$$
for any $\epsilon>0$, and thus $b_f(p)\neq 0$ always.
Theorem \ref{thm1.11} provides information on the non-lacunarity of the sequence $(b_f(p))$. 
\begin{corollary}
Suppose that $f(z)=\sum_{n=1}^{\infty} b_f(n) e^{{2\pi inz}}$ is a cusp newform of integer weight $k\geq 2$, level $N$, and character $\chi$. Then there exists  $0<\nu_f<1$ such that for $\theta>\nu_f$ and $Y\geq X^\theta$ we have $$\#\{X<p\leq X+Y;~b_f(p)\neq 0\} \gg_{f, \theta} \frac{Y}{\log{X}}.$$
In particular
$$j_f(p) \ll _{f, \theta} p^{\theta}.$$
\end{corollary} 
\noindent More specifically this corollary has the following consequence related to Lehmer's conjecture.
There is a constant $\theta$ with $0<\theta<1$ such that
\begin{equation}\label{serr2}
\#\{X<p\leq X+X^\theta;~\tau(p)\neq 0\} \gg_{\theta} \frac{X^\theta}{\log{X}}.
\end{equation}
In \cite{Moreno}, Moreno proved (\ref{serr2}) under the assumptions that $L(s, \piup_\Delta\times\piup_\Delta)$ satisfies a suitable explicit formula, a certain zero-free region, a specific upper bound for the number of its zeroes in a box, and a log-free zero-density estimate. Moreno calls the collection of these four assumptions the Hoheisel property of $L(s, \piup_\Delta\times\piup_\Delta)$.

The structure of the paper is as follows. In Section 2 we prove several preparatory lemmas. The proof of our main log-free density result (Theorem \ref{thm3}) follows closely the proof of the main theorem of \cite{Fogels}, which itself is based on Tur\'{a}n's power-sum method. In Section 3 we describe this method as stated by Fogels in his pole detection lemma (Lemma \ref{Fogels pole}). Section 4 is dedicated to a detailed proof of Theorem \ref{thm3}. Hypothesis \ref{HypothesisA} plays a crucial role in applying the  power-sum method in the proof of Theorem \ref{thm3}. 
In the final section we prove Theorems \ref{thm4} and \ref{thm1.11} and their corollaries. The proof of Theorem \ref{thm4} is an adaptation of the proof of Hoheisel's theorem to the case of automorphic $L$-functions. One important new feature is an application of a recent version of Perron's formula proved by Liu and Ye \cite{LY2}, which admits the use of Hypothesis \ref{HypothesisA}, which is weaker than GRC,  instead of GRC itself.

\begin{notation}
For two functions $f(x)$ and $g(x)\neq 0$, we use the notation $f(x)=O(g(x))$ for $x\in X$, or alternatively $f(x)\ll g(x)$ for $x\in X$ (where the set $X$ is specified either explicitly or implicitly), if $|f(x)/g(x)|$ is  bounded on $X$. We use the notation $f(x)\asymp g(x)$ if $f(x)\ll g(x)$ and $g(x)\ll f(x)$. Sometimes we write $f(x)\ll_t g(x)$ or $f(x)\asymp_t  g(x)$ when the implicit constants depend on the parameter $t$.
\end{notation}

\section{Preliminaries}
It is known that $s(1-s)\Phi(s, \piup\times \piup^\prime)=s(1-s)L(s,\piup_\infty\times \piup_\infty^\prime)L(s,\piup\times \piup^\prime)$ is an entire function of order one and moreover it has infinitely many zeroes. Thus, by the Hadamard factorization theorem, it may be written as an infinite product over its zeroes.  We denote a zero of $L(s, \piup\times\piup^\prime)$ by $\rho_{\piup\times \piup^\prime}$. We call $\rho_{\piup\times \piup^\prime}$ a \emph{trivial} zero if 
 $0\neq \rho_{\piup\times \piup^\prime}=-2n-\mu_{\piup\times\piup^\prime}(j, k)$ for some non-negative integer $n$, $1\leq j \leq m$, and $1\leq k \leq m^\prime$.
The functional equation \eqref{funceg} shows that all other zeroes are located in the critical strip $0\leq \Re(s) \leq 1$. We know that 
\begin{equation}
\label{trivial}
\log_p|\alpha_{\piup\times\tilde{\piup}}(j, k,p)|~{\rm and}~ |\Re(\mu_{\piup\times\piup^\prime}(j, k))|\leq 1-\frac{1}{m^2+1}-\frac{1}{(m^\prime)^2+1},
\end{equation}
(see \cite[formulas (4) and (6)]{B}).
This shows that the number of the trivial zeroes in the critical strip is finite. 
In the rest of the paper we use $\rho$ for a zero (trivial or non-trivial) of $L(s,\piup\times \piup^\prime)$. We denote the collection of zeroes (including multiplicity) of $L(s,\piup\times \piup^\prime)$ by $Z_{\piup\times\piup^\prime}$. We also note that by employing the bound \eqref{trivial} for $\alpha_{\piup\times\tilde{\piup}}(j, k, p)$ we can deduce, under the assumption of Hypothesis \ref{HypothesisA}, that  there exists ${\epsilon_\piup}^\prime$ such that  
\begin{equation*}
\sum_{X\leq n\leq X+Y} \Lambda(n) |a_{\piup\times\tilde{\piup}}(n)| \ll_\piup Y
\end{equation*}
whenever $Y\geq X^{1-{{\epsilon_\piup}^\prime}}$. Note that  
$a_{\piup\times\tilde{\piup}}(n)=|a_{\piup}(n)|^2$ except for finitely many primes.

In the following two lemmas we summarise some basic properties of $L(s, \piup\times\piup^\prime)$ and its zeroes.

\begin{Lem}
\label{LYlemma}
Let
\begin{equation*}
N_{\piup\times\piup^\prime}(T) = \# \{\rho=\beta+i\gamma; \, L_{\piup\times\piup^\prime}(\rho)=0~{\rm and}~|\gamma| \leq T\}
\end{equation*}
denote the zero-counting function of $L(s, \piup\times\piup^\prime)$.

\noindent
(a) Let $T>2$. Then $$N_{\piup\times \piup^\prime}(T+1)-N_{\piup\times \piup^\prime}(T) \ll \log{(Q_{\piup \times \piup^\prime} T)}$$ and
$$N_{\piup\times \piup^\prime}(T) \ll T\log{(Q_{\piup \times \piup^\prime} T)}.$$
(b) Let $s=\sigma+it$ with $-1/(m^2+1)-1/((m^\prime)^2+1)\leq \sigma \leq 2$. If  $s\not\in Z_{\piup\times\piup^\prime}$ and also $s\neq 0, 1$ if $\piup^\prime=\tilde{\piup}$, then
\begin{equation}\label{372}
\frac{L'}{L}(s, \piup\times\piup^\prime) +\frac{r_0}{s}+\frac{r_0}{s-1}- \sum_{\substack{{\rho}\\{|s-\rho|<1}}} \frac{1}{s-\rho} \ll  \log Q_{\piup\times\piup^\prime}( |t|+2),
\end{equation}
where $r_0=0$ if $\piup^\prime\neq\tilde{\piup}$ and $r_0=1$ if $\piup^\prime=\tilde{\piup}$. The terms in the sum in (\ref{372}) are repeated according to the multiplicity of $\rho$.
\end{Lem}
\begin{proof}
The proofs are standard, and analogous to the classical proofs for Dirichlet $L$-functions. See \cite[Theorem 5.8]{IK} and \cite[Theorem 5.7 (b)]{IK} for details. 
\end{proof}

\begin{Lem}
\label{LYlemma2}
(a) For $1<\sigma\leq 2$ we have $$\sum_{n=1}^{\infty} \frac{\Lambda(n) |a_\piup(n)|^2}{n^\sigma} \ll \frac{1}{\sigma-1}.$$
(b) For $X>1$ we have
$$\sum_{n\leq X} \frac{\Lambda(n) |a_\piup(n)|^2}{n} \ll \log{X}.$$
\end{Lem}
\begin{proof}
See \cite[(6.4)]{LY2} for a proof of (a).
The bound in (b) is a consequence of \eqref{PNT} and partial summation.
\end{proof}

The next two lemmas show that $L(s, \piup\times\piup^\prime)$ has the necessary properties for application of Tur\'{a}n's power-sum method (see Lemma \ref{Fogels pole}). The first result is a version of  Linnik's ``density lemma" (\cite[p.\ 331]{Prachar}) for $L(s, \piup\times\piup')$.

\begin{Lem}[\bf Density Lemma]
\label{LinLem}
For $t_0\in \mathbb{R}$ and $r>0$ let $ \nu(t_0, r, \piup\times\piup^\prime)$ be the number of zeroes of $L(s, \piup\times\piup^\prime)$ lying in the disk $G(t_0, r) = \{s\in \mathbb{C};~|s - (1 + it_0)|\leq r$\}. Then for $1/\log Q_{\piup\times\piup^\prime}(|t_0|+2) \ll r \leq 2$ we have
\begin{equation*}\label{LinLemEq}
\nu(t_0, r, \piup\times\piup^\prime) \ll r \log Q_{\piup\times\piup^\prime}(|t_0| + 2).
\end{equation*}
\end{Lem}
\begin{proof}
The lemma follows at once from part (a) of Lemma \ref{LYlemma} whenever $r\gg 1$. Accordingly choose $r< \frac{1}{2}$, whence the region $|s - (1 + it_0)| \leq r$ does not include the origin. Now we take the real part of $\frac{L'}{L}(s, \piup\times\piup^\prime)$ in (\ref{372}) and use the fact that $a_{\piup\times{{\piup}}^\prime}(n)=a_{\piup}(n) a_{{\piup}^\prime} (n)$ for $(q_{\piup}, q_{\piup^\prime})=1$. Then,  by part (a) of Lemma \ref{LYlemma2} and the Cauchy--Schwartz inequality, we have that
\begin{eqnarray*}
\Re \left(\frac{L'}{L}(s, \piup\times\piup^\prime)\right) \leq \bigg| \frac{L'}{L}(s, \piup\times\piup^\prime)\bigg| &\leq& 
\sum_{\substack{{p\mid (q_\piup, q_{\piup'})}\\{j, k}}} \frac{|\alpha_{\piup\times\piup^\prime}(j, k, p)|\log{p}}{|p^s-\alpha_{\piup\times\piup^\prime}(j, k, p) |}+\sum_{\substack{{n=1}\\{(n, (q_\piup, q_{\piup'}))=1}}}^{\infty} \frac{\Lambda(n) |a_\piup(n)||a_{\piup^\prime}(n)|}{n^\sigma} \\
&\ll&1+ \left\{ \sum_{n=1}^{\infty} \frac{\Lambda(n) |a_\piup(n)|^2 }{n^\sigma} \right\}^{1/2}  \left\{  \sum_{n=1}^{\infty} \frac{\Lambda(n)  |{a_{\piup^\prime(n)}}|^2}{n^\sigma}\right\}^{1/2}\\
&\ll& \frac{1}{\sigma - 1},
\end{eqnarray*}
when $1<\sigma\leq 2$.
By employing this inequality when $s_0 = 1+ r + it_0$, we can write (\ref{372}) as
\begin{equation}\label{loe}
\Re \left(\sum_{\substack{{\rho}\\{|s_0- \rho| < 1}}} \frac{1}{s_0-\rho} \right)\ll \frac{1}{r} + \log Q_{\piup \times \piup^\prime}(|t_0| + 2).
\end{equation}
It is easy to show that each of the summands in (\ref{loe}) is non-negative, indeed it follows that
\begin{equation}\label{loe2}
\Re \left(\sum_{\substack{{\rho}\\{|s_0-\rho| < 1}}} \frac{1}{s_0-\rho}\right) \geq \Re\left( \sum_{\rho \in G(t_0, r)} \frac{1}{s_0-\rho} \right)\geq \nu(t_0, r, \piup\times\piup^\prime) \min_{\rho\in G(t_0, r)} \frac{\Re(s_0- \rho)}{|s_0-\rho|^{2}},
\end{equation}
where in the first inequality we used the fact that $r<1/2$. 
Now observe that  if $\rho = \beta + i\gamma$, then, since $\rho\in G(t_0, r)$,
\begin{equation*}
\frac{\Re(s_0- \rho)}{|s_0-\rho|^{2}} = \frac{1+r  - \beta}{(r+1 - \beta)^2 + (t_0 - \gamma)^2} 
\geq \frac{1  - \beta}{2r^{2} + 2r(1-\beta)} = \frac{1}{2r}.
\end{equation*}
Thus (\ref{loe}) and (\ref{loe2}) show that
\begin{equation*}
\nu(t_0, r, \piup\times\piup^\prime) \ll 1 + r\log Q_{\piup\times\piup^\prime}(|t_0| + 2).
\end{equation*}
The lemma follows upon taking $r\gg 1/\log Q_{\piup\times\piup^\prime}(|t_0|+2)$.
\end{proof}

\begin{Lem}
\label{SecondLem}
Let $0< \theta_{0} \leq 1$ and suppose that ${L}(s, \piup\times\piup^\prime)$ has zeroes $\rho=\beta+i\gamma$ in the half plane $\sigma> 1- \theta_{0}$. Then, for any real $t_{0}$, and for $s$ in the disk $|s-(1+it_{0})| < \theta_{0}/2$, we have
\begin{equation}
\label{explicit}
 \frac{L'}{L}(s, \piup\times\piup^\prime)+\frac{r_0}{s}+\frac{r_0}{s-1} - \sum_{\substack{{\rho}\\{|\rho - (1 + it_{0})| < \theta_{0}}}} \frac{1}{s-\rho} \ll \log Q_{\piup \times \piup^\prime}(|t_{0}|+2).
\end{equation}
Here $r_0=0$ if $\piup^\prime\neq\tilde{\piup}$ and $r_0=1$ if $\piup^\prime=\tilde{\piup}$.
\end{Lem}
\begin{proof}
Let $s$ be fixed in the disk $|s-(1+it_{0})| < \theta_{0}/2$. Recall from \eqref{372} that
\begin{equation}
\label{new2.4}
\frac{L'}{L}(s, \piup\times\piup^\prime)+\frac{r_0}{s}+\frac{r_0}{s-1} = \sum_{\substack{{\rho}\\{|s-\rho|<3\theta_0/2}}} \frac{1}{s-\rho}  + O\left( \log Q_{\piup\times\piup^\prime}( |t|+2)\right).
\end{equation}
Note that
\begin{equation}\label{j24}
\sum_{\substack{{\rho}\\{|s-\rho| <3\theta_0/2}}}\frac{1}{s-\rho} = \sum_{\substack{{\rho}\\{|s-\rho| <3\theta_0/2}\\{|\rho - (1+ it_{0})}| < \theta_{0}}} \frac{1}{s-\rho}+ \sum_{\substack{{\rho}\\{|s-\rho| <3\theta_0/2}\\{\theta_{0}\leq |\rho-(1+it_{0})|< 2 \theta_{0}}}} \frac{1}{s-\rho}.
\end{equation}
Given that (\ref{new2.4}) is sought for $|s-(1+it_{0})|<\theta_{0}/2$ it follows that $|t|\leq |t_0|+\theta_0/2$ and also,   in the second sum on the right-hand side of (\ref{j24}), $|s-\rho| > \theta_{0}/2$. Thus $|s-\rho|$ is bounded away from zero
and each of its summands is $O(1)$ and moreover, by Lemma \ref{LinLem}, there are $\ll \log Q_{\piup\times\piup^\prime}(|t_0|+2)$ summands. This proves (\ref{explicit}).
\end{proof}

The next lemma establishes a standard zero-free region for $L(s, \piup\times\piup^\prime)$ under certain conditions.

\begin{Lem}[\bf Zero-free Region]
\label{zerofree}
Suppose that $m=m^\prime=2$ or at least one of $\piup$ or $\piup^\prime$ is self-dual. Then there exists a constant $c_0>0$, depending only on $m$ and $m^\prime$, such that in the region
$$\sigma>1-\frac{c_0}{\log{Q_{\piup\times\piup^\prime}(|t|+2)}}$$
we have $L(s, \piup\times\piup^\prime)\neq 0$ with at most one simple real exceptional zero $\beta_0<1$. For this real $\beta_0$ to exist, it is necessary that both $\piup$ and $\piup^\prime$ be self-dual.
\end{Lem}
\begin{proof}
See \cite[Theorem 5.44 (1)]{IK}, \cite[Theorem 3.3]{M2}, and \cite[p.\ 92]{GLS}.
\end{proof}

For $V$ distinct real numbers $w_1<w_2<\cdots<w_V$, set
\begin{equation}
\label{Smn}
S(m, n)=\sum_{1\leq j \leq V} \left( \frac{m}{n} \right)^{iw_j}.
\end{equation}
Estimating the size of $S(m, n)$ for certain values of $w_j$ is  important in the proof of Theorem \ref{thm3}. We also introduce the closely related sum
\begin{equation}
\label{gmn}
g(\mu, n)=\sum_{1\leq j\leq V} e^{iw_j(\mu-\log{n})},
\end{equation}
where $\mu$ is a real parameter.
It is useful to record the following two lemmas regarding $S(m, n)$ and $g(\mu, n)$.

\begin{Lem}\label{Sarastro}
Suppose that for a positive integer $m$ there exists a real number $\mu_m$ such that
\begin{equation*}
|\log m - \mu_{m}| \ll V^{\delta-1}~~~{\rm and}~~~ g(\mu_m, n)\leq V^\delta,
\end{equation*}
where $0\leq \delta<1$ and $V>0$. Then we have 
$S({m, n}) \ll V^{\delta}$. \end{Lem}
\begin{proof}
If $c$ is the implied constant in $|\log m - \mu_{m}| \ll V^{\delta-1}$ and  $|\theta_m|\leq 1$ is the real number such that $\log{m}=\mu_m+c\theta_m V^{\delta-1}$, then
\begin{equation*}
\begin{split}
S({m, n}) &= \sum_{1\leq j \leq V} \exp \{ iw_{j}[\mu_{m} + {c\theta_m}{V^{\delta-1}} - \log n]\}\\
&=  \sum_{1\leq j \leq V} \exp \{ iw_{j}[\mu_{m} - \log n]\} + O\left(\sum_{1\leq j \leq V} V^{\delta-1}\right) \ll V^{\delta}.
\end{split}
\end{equation*}
\end{proof}

\begin{Lem}\label{Pamina}
Let $T$ and $\lambda$ be real numbers and let $c^\prime$ be a fixed positive constant.  Let $0<w_1<w_2<\cdots<w_V\leq 2T$ be real numbers such that each $w_i$ is an integer multiple of $c^\prime/\log{T}$, and for each $w_i$ there is an integer $m$ such that $w_i\in [\frac{m\lambda}{\log{T}}, \frac{(m+1)\lambda}{\log{T}}]$.  Suppose that  $V\ll \min\{e^{c\lambda} \log{T}, \log^2{T}\}$ for some positive constant $c$, where $V$ is the number of $w_i$'s in the interval $[0, 2T]$. Then, for any $0\leq \delta<1$, we have
$$\texttt{meas}\{ \mu\in [\mu_1, \mu_2];~|g(\mu, n)|>V^\delta\} \ll V^{1-2\delta}(\mu_2-\mu_1+e^{c\lambda}\log{T}),$$
where $\texttt{meas}$ denotes the Lebesgue measure.
\end{Lem}
\begin{proof} We have
\begin{eqnarray}\label{alm1}
\int_{\mu_{1}}^{\mu_{2}} |g(\mu)|^{2}\, d\mu &=& \int_{\mu_{1}}^{\mu_{2}} \sum_{1\leq j, j'\leq V} e^{i(w_{j} - w_{j'})(\mu - \log n)}\, d\mu \nonumber\\
&= &\int_{\mu_{1} - \log n}^{\mu_{1} - \log n}\sum_{1\leq j, j'\leq V} e^{i(w_{j} - w_{j'})\phi}\, d\phi \nonumber\\
&\ll& \sum_{1\leq j \leq V} (\mu_{2} - \mu_{1}) + \sum_{\substack{1\leq j, j'\leq V\\ j\neq j'}} \frac{1}{|w_{j} - w_{j'}|}\nonumber\\
&\ll& V(\mu_{2} - \mu_{1}) + \sum_{\substack{1\leq j' \leq V\\ j>j'}} \frac{1}{w_{j} - w_{j'}}.
\end{eqnarray}
One may estimate the final sum in (\ref{alm1}) using Stieltjes integrals, whence
\begin{equation*}
\sum_{j> j'} \frac{1}{w_{j} - w_{j'}} \ll \log{T} \left( \left[ \frac{\log^2{u}}{u}\right]_{1}^{\infty}+e^{c\lambda} \int_{1}^{\infty}\frac{\log u}{u^{2}}\, du \right)\ll e^{c\lambda}\log{T}.
\end{equation*}
(Note that $V\ll \min\{e^{c\lambda} \log{T}, \log^2{T}\}$.) Thus (\ref{alm1}) shows that $\int_{\mu_{1}}^{\mu_{2}} |g(\mu)|^{2}\, d\mu \ll V(\mu_2-\mu_1+e^{c\lambda}\log{T})$ and the lemma follows. 
\end{proof}

\section{Fogels' pole detection lemma}
In this section we describe a fundamental lemma in Fogels' method of establishing log-free zero-density estimates. The lemma itself is based on Tur\'{a}n's power-sum method.

Let $D>2$, $0<\theta_0\leq 1$, and $c_0>0$. Let $F(s)$  be a meromorphic function defined on $\sigma>1-\theta_0$ with simple poles $\rho=\beta+i\gamma$ of positive residue $m_\rho$ which lie in $1-\theta_0<\sigma\leq 1$ and possibly a simple pole of residue $-1$ at $s=1$. 
Moreover assume the truth of the following two statements for $F(s)$.

(i) If $\nu(t_{0}, r, F)$ denotes the number of poles of $F(s)$ in the disk $|s-(1+t_{0})| \leq r$, then 
\begin{equation}\label{Linnikd}
\nu(t_0, r, F) \ll r \log D( |t_{0}|+2),
\end{equation}
where $ c_{0}/\log D(|t_{0}|+2) \leq r \leq \theta_0$. 

(ii) For 
any real $t_{0}$ we have
\begin{equation}\label{372.1}
 F(s)+\frac{r_0}{s-1} - \sum_{\substack{{\rho}\\{|\rho - (1 + it_{0})| < \theta_{0}}}} \frac{m_\rho}{s-\rho} \ll \log D(|t_{0}|+2), \end{equation}
for $|s-(1+it_{0})| < \theta_{0}/2$, where $r_0=1$ if $F(s)$ has a pole at $s=1$ and $r_0=0$ otherwise.

\noindent For $\lambda\in [0, \log{D}]$ and $-D\leq t_0 \leq D$ we consider the square
$$S(\lambda, t_0, D)=\{s=\sigma+it;\, 1- \lambda/\log D\leq \sigma \leq 1, \, |t-t_{0}|\leq \lambda /2\log D\}.$$
For $c_0\leq \lambda \leq \log{D}$, set $A=\lambda^{-1} \log{D}$. Also for an integer $k\geq 2$, and real $\tau$, we define
\begin{equation}\label{actJdef}
J(\tau, k, A) = -\frac{1}{2\piup i}\int_{2-i\infty}^{2+ i\infty} \left(\frac{e^{3As} - e^{As}}{2As}\right)^{k}F(1+ s + i\tau)\, ds.
\end{equation}

Under the above conditions Fogels proved the following lemma.

\begin{Lem}[\bf Fogels]
\label{Fogels pole}
There are positive constants $E$, $c_1~(\leq \theta_0)$, $b_0$, and an integer $k\geq 2$ such that for $\lambda \in [c_0, c_1\log{D}]$, if there is a pole of $F(s)$ in   $S(\lambda, t_0, D)$, then for all $\tau\in [t_0-\lambda/2\log{D}, t_0+\lambda/2\log{D}]$ we have
\begin{equation*}
|J(\tau, k, \lambda^{-1}\log D)| > e^{-b_0\lambda},
\end{equation*}
as long as $|t_0|\geq E\lambda/\log{D}$.
Moreover $\alpha_2 \lambda\leq k <(\alpha_1+\alpha_2)\lambda$, where $\alpha_1>1$ is an appropriate constant and $\alpha_2\geq \alpha_1^2$ can be chosen arbitrarily.  The constant $c_1$ depends on $\theta_0$ and the constant $b_0$ may depend on $\alpha_1$ and $\alpha_2$.
\end{Lem}
\begin{proof}
See \cite[Lemma 16]{Fogels2} and \cite[Section 6, p. 75]{Fogels}.
\end{proof}
Note that in Lemma \ref{Fogels pole}  $\alpha_2$ can be chosen arbitrarily large. We shall exploit this fact when we use Hypothesis \ref{HypothesisA} in the proof of Theorem \ref{thm3}.

\section{Proof of the log-free density theorem}
\subsection{Basic set up}  We shall prove the following equivalent version of Theorem \ref{thm3} by  following closely
the method devised by Fogels in \cite{Fogels}.
\medskip\par
\noindent{\bf Theorem 1.2} {\it Assume that Hypothesis \ref{HypothesisA} holds for $\piup$ and $\piup^\prime$ and that either $m, m^\prime\leq 2$ or at least one of $\piup$ and $\piup^\prime$ is self-dual. Then there exist positive constants $c$ and $T_0$ (depending on $\piup$ and $\piup^\prime$) such that for any $T\geq T_0\geq \max\{3, Q_{\piup\times\piup^\prime}\}$ and $\lambda\in [0, \log{T}]$, we have}
$$N_{\piup\times\piup^\prime}(1-\lambda/\log{T}, T)\leq e^{c\lambda}.$$

We start by considering 
$$G(s)= \prod_{p\mid (q_{\piup}, q_{\piup^\prime})} \prod_{\substack{{j,k}\\{1\leq j \leq m}\\{1\leq k\leq m^\prime}}}\left(1-\frac{\alpha_{\piup\times\piup^\prime}(j,k, p)}{p^s}  \right)L(s,\piup\times\piup^\prime).$$
We observe that $N_{ \piup\times\piup^\prime}(\sigma, T)\leq N_{G }(\sigma, T)$, whence the result follows by showing that  $N_{G}(1-\lambda/\log{T}, T)\leq e^{c\lambda},$ for a suitable positive $c$.

We continue by bounding $N_{G}(1- \lambda/\log T, T)$, the number of zeroes of $G(s)$ in the rectangle $$R(\lambda, T) = \{s=\sigma+it:\, 1- \lambda/\log T \leq \sigma \leq 1, \, -T \leq t \leq T\},$$ by covering $R(\lambda, T)$ with squares
\begin{equation*}
S(\lambda, t_{0}, T) =  \{s=\sigma+it:\, 1- \lambda/\log T \leq \sigma \leq 1, \, |t-t_{0}| \leq \lambda/2\log T\},
\end{equation*}
and applying Lemma \ref{Fogels pole} to each square $S(\lambda, t_{0}, T)$ with sufficiently large $t_0$. 

\subsection{Application of Fogels' lemma} First, note that the poles of 
$\frac{G^\prime}{G}(s)$ occur at the zeroes of $G(s)$ and possibly at $s=1$.  Moreover the poles at $s\neq 1$ have positive residues and the possible pole at $s=1$ has residue $-1$. By Lemmas \ref{LinLem} and \ref{SecondLem}, $\frac{G^\prime}{G}(s)$ satisfies \eqref{Linnikd} and \eqref{372.1} if $(q_\piup, q_{\piup^\prime})=1$. Using the explicit expression for zeroes of $$\prod_{p\mid (q_{\piup}, q_{\piup^\prime})} \prod_{\substack{{j,k}\\{1\leq j \leq m}\\{1\leq k\leq m^\prime}}}\left(1-\frac{\alpha_{\piup\times\piup^\prime}(j,k, p)}{p^s}  \right)$$ we can show results analogous to Lemma \ref{LYlemma} hold for $G$ and therefore results analogous to Lemmas \ref{LinLem} and \ref{SecondLem} also hold. Hence $\frac{G^\prime}{G}(s)$ is suitable for the application of Fogels' lemma.  Second, by adjusting the constant in the zero-free region obtained in Lemma \ref{zerofree} we can find a constant $c_0$ such that  $N_G(1- c_0/\log T, T)\leq 1$ for $T\geq Q_{\piup\times\piup^\prime}$. Therefore we can apply Lemma \ref{Fogels pole} with $F(s)=\frac{G^\prime}{G}(s)$, $c_0$, $D=T$, $\theta_0=1$, an appropriate $\alpha_1$, and $\alpha_2\geq \max\{\alpha_1^2, 2/\epsilon_\piup, 2/\epsilon_{\piup^\prime}\}$. Hence there are positive constants $c_1\leq 1$, $b_0$ (depending only on $\alpha_1$ and $\alpha_2$), and an integer $k\geq 2$ ($\alpha_2\leq k\lambda^{-1}<\alpha_1+\alpha_2$) such that for $\lambda\in [c_0, c_1\log{T}]$ we have
\begin{equation}
\label{Jdef} 
|J(\tau, k, \lambda^{-1}\log T| > e^{-b_0\lambda},
\end{equation}
whenever the square $S(\lambda, t_0, T)$ contains a zero of $G(s)$ and $|t_0|>E\lambda/\log{T}$. (Recall that $E$ is given in Lemma \ref{Fogels pole} and $\tau$ is any number in the interval $[t_0-\lambda/2\log{T}, t_0+\lambda/2\log{T}]$.) Note that if $\lambda<c_0$ we have $N_G(1- \lambda/\log T, T)\leq 1$. Also, if $\lambda>c_1\log{T}$ by Lemma \ref{LYlemma}(a) the theorem follows since
$$N_G(1- \lambda/\log T, T) \ll T^2 \ll e^{(2/c_1)\lambda},$$
for $T$ sufficiently large.
Henceforth in order to prove Theorem \ref{thm3} we only need to assume that $\lambda\in [c_0, c_1\log{T}]$.

Next we observe that $$R(\lambda, T)=\bigcup_{0\leq m \leq [2T\log T/\lambda]} S(\lambda, t_m, T),~{\rm where}~t_m=-T + \left(m+ {1}/{2}\right)\lambda/\log T.$$
Note that $t_{m} = -T + (m+ \frac{1}{2})\lambda/\log T$ the midpoints on the squares are, where $0\leq m\leq [2T\log T/\lambda]$. Select a number $\tau\in [t_m-\lambda/2\log{T}, t_m+\lambda/2\log{T}]$ from each square $S(\lambda, t_m, T)$ to act as the representative of that square. More precisely choose a representative  $\tau$ that differs from $-T$ by an integer multiple of $c'/\log T$, where $c'< c_{0} $ is a sufficiently small constant (the size of which will be specified later). 
On each square with $|t_m|>E\lambda/\log{T}$ that contains a zero of $G(s)$ the result (\ref{Jdef}) now applies.  

Now consider all the values of $k$ that appear in the squares $S(\lambda, t_{m}, T)$ with $|t_m|>E\lambda/\log{T}$ after application of Lemma \ref{Fogels pole}. Let $k_{0}$ be the most frequently occurring value of $k$ (or one of the most frequently occurring values) in (\ref{Jdef}), and suppose $k_{0}$ appears $V$ times. Then the $\tau$'s chosen from the boxes in which $k_{0}$ appears can be denoted by
\begin{equation}\label{taujdef}
\tau_{j} = -T + w_{j}, \quad\quad 1\leq j\leq V.
\end{equation}

It is now relatively easy to estimate the total number of zeroes in $R(\lambda, T)$. By Lemma \ref{LinLem}, the number of zeroes in a single box $S(\lambda, t_{m}, T)$ is $\ll \frac{\lambda}{\log T} \log T(|t_{m}|+2) \ll \lambda$. There are finitely many boxes $S(\lambda, t_{m}, T)$ for which $|t_m|\leq E\lambda/\log{T}$, so the total number of zeroes in such boxes is bounded by $\lambda$. For boxes for which $|t_m|> E\lambda/\log{T}$, there are $V$ boxes in which $k_{0}$ appears. Since the number of choices of $k$ is $\ll \lambda$ (recall that $\alpha_2<k\lambda^{-1}<\alpha_1+\alpha_2$) then
\begin{equation}\label{firstden}
N_{G}(1-\lambda/\log T, T) \ll \lambda+\lambda^{2} V.
\end{equation}
We next find an estimation for $V$.

\subsection{An estimate on \texorpdfstring{$V$}{Lg}}\label{Vsec}
For $\sigma>1$
\begin{equation*}\label{logder}
-\frac{G'}{G}(s) = \sum_{n=1}^{\infty} \frac{\Lambda(n)a_{G}(n)}{n^{s}},
\end{equation*}
where 
$$a_G(n)=\left\{ \begin{array}{cc} a_\piup(n) a_{\piup^\prime}(n)&{\rm if}~(n, (q_\piup, q_{\piup^\prime}))=1,\\ 0&{\rm otherwise.}\end{array} \right.$$
This, together with (\ref{actJdef}), shows that
\begin{equation*}
J(\tau, k, A) = \sum_{n=1}^{\infty} \frac{\Lambda(n)a_{G}(n)}{n^{1+ i\tau}}R(n, k, A),
\end{equation*}
where 
\begin{equation*}
R(n, k, A) = \frac{1}{2\piup i} \int_{2-i\infty}^{2+ i\infty}\left(\frac{e^{3As} - e^{As}}{2As}\right)^{k} n^{-s}\, ds.
\end{equation*}
(Recall that $A=\lambda^{-1} \log{T}$.) From \cite[Lemma 17]{Fogels2} we know that
\begin{equation}\label{Rsize}
|R(n, k, A)| \leq \begin{cases} e^{c_2 k}/A &\mbox{if } e^{kA} < n < e^{3kA}, \\
0 &\mbox{otherwise},
\end{cases}
\end{equation}
where $c_2>1$.
After setting $B=k_0\lambda^{-1}$ it follows from (\ref{Jdef}) that, for $k_0$ and $\tau_j$ ($1\leq j \leq V$) defined above,
\begin{equation*}\label{146}
\bigg| \sum_{T^{B} < n < T^{3B}} \frac{\Lambda(n) a_{G}(n)R(n, k_0, A)}{n^{1+ i\tau_j}}\bigg| > e^{-b_0\lambda},
\end{equation*}
whence
\begin{equation*}
\sum_{1 \leq j \leq V} \left| \sum_{T^{B} < n < T^{3B}} \frac{\Lambda(n)a_{G}(n) R(n, k_0, A)}{n^{1+ i\tau_j}} \right|> Ve^{-b_0\lambda}.
\end{equation*}
An application of the Cauchy--Schwartz inequality and the invocation of (\ref{taujdef})
show that
\begin{equation*}
\begin{split}
V &\leq e^{2b_0\lambda} \sum_{T^{B} < n < T^{3B}} \frac{\Lambda(n)a_{G}(n) R(n, k_0, A)}{n^{1 - iT}} \sum_{T^{B} < m < T^{3B}} \frac{\Lambda(m)\overline{a_{G}(m)}~ \overline{R(m, k_0, A)}}{m^{1  +iT}}  \sum_{1 \leq j \leq V} \left(\frac{m}{n}\right)^{iw_{j}} \\
& \leq 2 e^{2b_0\lambda} \sum_{T^{B} < n \leq m < T^{3B}} \frac{\Lambda(n)\Lambda(m) | a_{G}(n) \overline{a_{G}(m)}| |R(n, k_0, A)\overline{R(m, k_0, A)}|}{nm} \bigg| \sum_{1 \leq j \leq V} \left(\frac{m}{n}\right)^{iw_{j}} \bigg|.\\
\end{split}
\end{equation*}
By employing the bound \eqref{Rsize} in the above inequality we deduce that
\begin{equation}\label{Fog20}
V\leq \frac{\lambda^{2}e^{c_3\lambda}}{\log^2 T} \sum_{T^{B} < n \leq m < T^{3B}} \frac{\Lambda(n)\Lambda(m)|a_{G}(n)||a_{G}(m)|}{nm} \bigg| \sum_{1 \leq j \leq V} \left(\frac{m}{n}\right)^{iw_{j}} \bigg|.
\end{equation}
In (\ref{Fog20}) the numbers $w_{j}$ run over $V$ terms (those corresponding to the integer $k_{0}$). If instead the $w_{j}$'s are allowed to run over \textit{all} integer multiples of $c'/\log T$ less than $2T$ then the final sum in (\ref{Fog20}) is a geometric progression.

For a fixed $n\in (T^{B}, T^{3B})$, write $\frac{m}{n} = e^{\mu}$, where $0\leq \mu < 2B \log T$. Write
\begin{equation*}
M_{\ell} = \left\{\mu;~ \frac{\ell}{2T} \leq \mu < \frac{(\ell+1)}{2T} \right\}, \quad\quad 0\leq \ell \ll T \log T,
\end{equation*}
so that the union of the $M_{\ell}$'s covers all possible sums arising in (\ref{Fog20}). Certainly the contribution to the sum $\sum_{1 \leq j \leq V}(m/n)^{iw_{j}}$ from the interval $M_{0}$ is $\ll T\log T$.
For $\ell\neq 0$ we use the estimate
\begin{equation}\label{Prach}
\sum_{n\leq N} e^{in\phi} \ll \frac{1}{\min\{\phi, 2\pi - \phi\}},
\end{equation}
for $0< \phi< 2\pi$. Recall that $w_{j} = jc'/\log T$ for $1 \leq j \ll T\log T$ and that $c'$ has to be smaller than $c_{0}$. If, in addition, $c'$ is sufficiently small to ensure that $\frac{c'}{\log T} 2B \log T <\pi$, then it follows that $w_{1}\mu < \frac{c'}{\log T} 2B \log T <\pi$. Therefore one may use (\ref{Prach}) to show that
\begin{equation}\label{lastw}
\sum_{1 \leq j \leq V} \left(\frac{m}{n}\right)^{iw_{j}} = \sum_{1\leq j \leq V} e^{\frac{i\mu j c'}{\log T}} \ll \left(\frac{c'}{\log T} \frac{\ell}{2T}\right)^{-1} \ll \frac{T\log T}{\ell}.
\end{equation}

Now if $m/n = e^{\mu} \in M_{\ell}$ then $ne^{\ell/2T} \leq m < n e^{\ell/2T} e^{1/2T}$, so that $m$ is in the interval $\mathcal{M}_{\ell} = [x_{\ell}, x_{\ell} \exp(1/2T))$, where $x_{\ell} = ne^{\ell/2T}$. We observe that since $B\geq \alpha_2\geq\max\{2/\epsilon_{\piup}, 2/\epsilon_{\piup^\prime}\}$ then 
$$x_\ell (e^{1/2T}-1)\geq \frac{x_\ell}{2T}\geq\max\{x_\ell^{1-\epsilon_{\piup}}, x_\ell^{1-\epsilon_{\piup^\prime}}\}$$ and thus, by the Cauchy--Schwartz inequality, Hypothesis \ref{HypothesisA} for $\piup$ and $\piup^\prime$ may be applied to show that
\begin{equation*}
\sum_{m \in \mathcal{M}_{\ell}} \frac{\Lambda(m)|a_{G}(m)|}{m}\leq \left\{\sum_{m \in \mathcal{M}_{\ell}} \frac{\Lambda(m)|a_{\piup}(m)|^2}{m}\right\}^{1/2} \left\{\sum_{m \in \mathcal{M}_{\ell}} \frac{\Lambda(m)|a_{\piup^\prime}(m)|^2}{m}\right\}^{1/2}\ll T^{-1},
\end{equation*}
whence, by (\ref{lastw})
\begin{equation*}
\sum_{m \in \mathcal{M}_{\ell}} \frac{\Lambda(m)|a_{G}(m)|}{m} \bigg| \sum_{1\leq j \leq V} \left(\frac{m}{n}\right)^{iw_{j}}\bigg| \ll \begin{cases} \log T & \mbox{if } \ell= 0\\
\ell^{-1}\log T & \mbox{if } 1 \leq \ell \ll T\log T.
\end{cases}
\end{equation*}
Adding these intervals gives
\begin{equation*}
\sum_{n \leq m < T^{3B}} \frac{\Lambda(m)|a_{G}(m)|}{m} \bigg| \sum_{1\leq j \leq V} \left(\frac{m}{n}\right)^{iw_{j}}\bigg| \ll \log^{2} T.
\end{equation*}
The above estimate together with part (b) of Lemma \ref{LYlemma2} and the Cauchy--Schwartz inequality yield
\begin{eqnarray}\label{lastsec1}
\sum_{T^{B} <n \leq m < T^{3B}} \frac{\Lambda(n)\Lambda(m)|a_{G}(n)||a_{G}(m)|}{nm} \bigg| \sum_{1\leq j \leq V} \left(\frac{m}{n}\right)^{iw_{j}}\bigg| 
&\ll& \log^{2}T \sum_{T^{B} < n < T^{3B}} \frac{\Lambda(n)|a_{G}(n)|}{n}\nonumber \\
&\ll& \log^{3}T.
\end{eqnarray}
Equations (\ref{lastsec1}) and (\ref{Fog20}) show that
\begin{equation}\label{24}
V \leq e^{c_4\lambda} \log T
\end{equation}
for some positive constant $c_4$.

\subsection{A refinement that gives a log-free result}\label{logfs}
In order to eliminate the factor of $\log T$ in \eqref{24} we assume that $V> e^{4c_{4}\lambda}$, since otherwise there is nothing to prove. It follows from this assumption that $V< \log^{2}T$.

We plan on estimating the sum in (\ref{lastsec1}), that is,
\begin{equation}\label{E}
\Sigma = \sum_{T^{B} <n \leq m < T^{3B}} \frac{\Lambda(n)\Lambda(m)|a_{G}(n)||a_{G}(m)|}{nm} |S(m,n)|,
\end{equation}
where $S(m,n)$ is defined in \eqref{Smn}.
We show that $\Sigma \ll V^{\eta} \log^{2}T$, for some positive $\eta<1$. It then follows from (\ref{Fog20}) that $V \ll e^{c_4\lambda} V^{\eta}$, where $c_4$ is the constant given in \eqref{24}, and so $V \ll e^{(c_4/(1-\eta))\lambda}$ which proves Theorem \ref{thm3}.

Fix an integer $m_{0} \in (T^{B}, T^{3B})$, let $M_{m_0} = [m_{0}, m_{0} \exp(V^{-1/4}\log T)]$ and let $I_{m_0}$ be the logarithm of the interval $M_{m_0}$ --- that is $I_{m_0} = [\log m_{0}, \log m_{0} + V^{-1/4}\log T]$. 
Now we cover the interval $(T^{B}, T^{3B})$ by intervals $M_{m_0}$; we need at most $2BV^{1/4}$ such intervals. Consider the intervals $I_{m_0}$ corresponding to the intervals $M_{m_0}$. Divide each $I_{m_0}$ into $[T\log T]$ equal parts such that each part has length at most $h/(TV^{1/4})$, where $1\leq h <2$. Label these intervals $I_{m_0, 1}, I_{m_0, 2}, \ldots, I_{m_0, [T\log T]}.$ We now consider all pairs $(m, n)$ such that $\log{m}\in I_{m_0, i}$ for some $m_0$, $i$, and for which there exists $\mu_m \in  I_{m_0, i}$ such that $|g(\mu_m, n)|\leq V^{7/8}$, where $g(\mu_m, n)$ is defined in \eqref{gmn}. We denote such a pair by $(m, n)^*$. 
Since $|\log{m}-\mu_m|<2/(TV^{1/4})$, Lemmas \ref{Sarastro} and \ref{LYlemma2}(b) give the following estimate for a part of (\ref{E})
\begin{equation}\label{Tamino}
\sum_{\substack{T^{B}<n\leq m < T^{3B} \\ (m, n)^{*}}} \frac{\Lambda(n)\Lambda(m)|a_{G}(n)||a_{G}(m)|}{nm} |S(m, n)| \ll V^{7/8}\log^{2}T.
\end{equation}

Now suppose that $(m, n)$ is not an $(m, n)^*$. For such $m$ we have $|g(\mu, n)|>V^{7/8}$ for all $\mu\in I_{m_0, i}$, where $\log{m}\in I_{m_0, i}$. We denote such $m$ by $m^{\dagger}(n)$. Now for fixed $n$ we have the following estimation.
\begin{equation}
\label{4.12}
\sum_{\substack{T^{B}< m < T^{3B} \\ m=m^{\dagger}(n)}} \frac{\Lambda(m)|a_{G}(m)|}{m}\leq \sum_{m_0,i}^{\prime} \sum_{\substack{T^{B}< m < T^{3B} \\ \log{m}\in I_{m_0, i}}} \frac{\Lambda(m)|a_{G}(m)|}{m},
\end{equation}
where $\displaystyle{\sum_{m_0,i}^{\prime}}$ is a sum over all intervals $I_{m_0, i}$ with the property that
$|g(\mu, n)|\geq V^{7/8}$ for $\mu\in I_{m_0, i}$. Let $x_0$ be such that if $x_0\leq m\leq x_0 \exp(h/TV^{1/4})$ then $\log{m}\in I_{m_0, i}$. 
Since  $B\geq\alpha_2\geq\max\{2/\epsilon_{\piup}, 2/\epsilon_{\piup^\prime}\}$, $1\leq h< 2$, and $V\leq \log^2{T}$,  then 
$$x_0 (e^{h/(TV^{1/4})}-1)\geq \frac{x_0 h}{TV^{1/4}}\geq\max\{x_0^{1-\epsilon_{\piup}}, x_0^{1-\epsilon_{\piup^\prime}}\},$$ and thus, by employing the Cauchy--Schwartz inequality, we may apply Hypothesis \ref{HypothesisA} for $\piup$ and $\piup^\prime$ to get
\begin{equation}
\label{4.13}
\sum_{\substack{T^{B}<m < T^{3B} \\ \log{m}\in I_{m_0, i}}} \frac{\Lambda(m)|a_{G}(m)|}{m}\leq \sum_{\substack{ T^B<x_0\leq m\leq x_0 \exp(h/TV^{1/4})}} \frac{\Lambda(m)|a_{G}(m)|}{m}\ll \frac{1}{TV^{1/4}}.
\end{equation}
From Lemma \ref{Pamina} and $e^{c_4\lambda}<V^{1/4}$ we conclude that 
\begin{eqnarray}
\label{4.14}
\#\{I_{m_0, i};~|g(\mu, n)|\geq V^{7/8}~{\rm for}~\mu\in I_{m_0, i}\}&\ll& (TV^{1/4})\left(V^{-3/4}(V^{-1/4} \log{T}+e^{c_4 \lambda} \log{T} )  \right)  \nonumber  \\
&\ll& TV^{-1/2}e^{c_4\lambda}\log{T}\nonumber\\
&<& TV^{-1/4} \log{T}.
\end{eqnarray}
Since the total number of $m_0$ is at most $2BV^{1/4}$, we can apply \eqref{4.13} and \eqref{4.14} to \eqref{4.12}, whence
\begin{equation*}
\sum_{\substack{T^{B}< m < T^{3B} \\ m=m^{\dagger}(n)}} \frac{\Lambda(m)|a_{G}(m)|}{m} \ll V^{-1/4} \log{T}.
\end{equation*}
Thus the remaining part of (\ref{E}) becomes
\begin{eqnarray}\label{E2}
\sum_{\substack{T^{B}<n\leq m < T^{3B} \\ (m, n)\neq (m, n)^{*}}} \cdots  &\ll& 
V \sum_{T^{B} < n < T^{3B}} \frac{\Lambda(n)|a_{G}(n)|}{n} \sum_{\substack{T^{B}\leq m < T^{3B} \\ m=m^{\dagger}(n)}} \frac{\Lambda(m)|a_{G}(m)|}{m}\nonumber\\
&\ll&V^{3/4} \log^2{T}.
\end{eqnarray}

Hence (\ref{Tamino}) and \eqref{E2} show that $\Sigma \ll V^{7/8}\log^{2}T$ as desired. This bound for $\Sigma$ in combination with (\ref{Fog20}) yields $V \ll e^{c_4\lambda} V^{7/8}$, whence $V \ll e^{8c_4\lambda}$. This, together with \eqref{firstden}, proves the theorem.


\section{Applications}
The following version of Perron's formula is proved in \cite[Theorem 2.1]{LY2}.
\begin{Lem}[{\bf Liu--Ye}]
\label{Liu-Ye}
Let $f(s)=\sum_{n=1}^{\infty} \frac{a_n}{n^s}$ be an absolutely convergent series in the half-plane $\sigma>\sigma_a$. Let $B(\sigma)=\sum_{n=1}^{\infty} \frac{|a_n|}{n^\sigma}$ for $\sigma>\sigma_a$. Then for $b>\sigma_a$, $X\geq 2$, $T\geq 2$, and $H\geq 2$,
\begin{eqnarray*}
\label{Perron}
\sum_{n\leq X} a_n &=&\frac{1}{2\pi i} \int_{b-iT}^{b+iT} f(s) \frac{X^s}{s}ds+
O\left( \sum_{X-X/H<n\leq X+X/H} |a_n|\right)+O\left( \frac{HX^bB(b)}{T} \right).
\end{eqnarray*}
\end{Lem}
This version of Perron's formula is useful when one has a suitable upper bound for the average of the coefficients $|a_n|$ over short intervals.
\begin{proof}
[Proof of Theorem \ref{thm4}]
An application of Lemma \ref{Liu-Ye} for $b=1+1/\log{X}$, $H=T^{\epsilon_\piup}$, and $$f(s)=-\frac{L^\prime}{L}(s, \piup\times\tilde{\piup})=\sum_{n=1}^{\infty} \frac{\Lambda(n) a_{\piup\times\tilde{\piup}}(n)}{n^s}$$
results in 
\begin{eqnarray}
\label{Perron2}
\psi_{\piup\times\tilde{\piup}}(X)&=& \frac{1}{2\pi i} \int_{b-iT}^{b+iT} -\frac{L^\prime}{L}(s, \piup\times\tilde{\piup}) \frac{X^s}{s}ds+O\left( \sum_{X-X/T^{\epsilon_\piup}<n\leq X+X/T^{\epsilon_\piup}} \Lambda(n) |a_{\piup\times \tilde{\piup}}(n)|\right)\nonumber\\
&&+O\left( \frac{X\log{X}}{T^{1-\epsilon_\piup}} \right).
\end{eqnarray}
(Note that by Lemma \ref{LYlemma2}(b) we have $B(\sigma)\ll 1/(\sigma-1)$ if $1<\sigma\leq 2$.) Now we assume that $1\leq T\leq X$, whence by Hypothesis \ref{HypothesisA} and the bound \eqref{trivial} for unramified primes, the first error term in the above formula is $O(X/T^{\epsilon_\piup})+O(X^{1-2/(m^2+1)}(\log{X})^2)$. Without loss of generality let us assume that $0<\epsilon_\piup\leq 1/2$. We can rewrite \eqref{Perron2} as
\begin{equation*}
\psi_{\piup\times\tilde{\piup}}(X)= \frac{1}{2\pi i} \int_{b-iT}^{b+iT} -\frac{L^\prime}{L}(s, \piup\times\tilde{\piup}) \frac{X^s}{s}ds+O\left( \frac{X\log{X}}{T^{\epsilon_\piup}}\right)+O\left(X^{1-{2}/{(m^2+1)}}(\log{X})^2 \right).
\end{equation*}
A standard argument involving moving the line of integration in the above formula to the half-plane $\Re(s)<0$ and computing the residues at $s=1$, $s=\rho$ (zeroes of $L(s, \piup\times\tilde{\piup})$), and $s=0$ (see \cite[Chapter 17]{D} for details) implies that for $2^{\frac{1}{\epsilon_\piup}}\leq T\leq X$ we have 
 \begin{equation}
\label{last} 
\psi_{\piup\times\tilde{\piup}}(X)= X-\sum_{\substack{{\rho=\beta+i\gamma}\\|\gamma|\leq T}} \frac{X^\rho}{\rho}+O\left( \frac{X\log{X}}{T^{\epsilon_\piup}}\right)+O(X^{\mu_0}),
\end{equation}
where $\mu_0$ can be taken as any number in the interval $(1-2/(m^2+1),1)$. The last error term comes from $\sum X^\mu/\mu$ where $\mu$ ranges over the trivial zeroes of $L(s, \piup\times\tilde{\piup})$. Note that by \eqref{trivial} this sum is bounded by $X^{\mu_0}$. 
Following the argument on \cite[p.\ 257]{Ingham1937} from the explicit formula \eqref{last} we derive (for $Y\leq X$)
\begin{equation}
\label{eformula}
\frac{1}{Y}\left( \psi_{\piup\times\tilde{\piup}}(X+Y)-\psi_{\piup\times\tilde{\piup}}(X)\right)= 1+O\left(X^{\beta_0-1}\right)+O\left( \sum_{\substack{{\rho\neq \beta_0}\\{|\gamma|\leq T}}} X^{\beta-1}\right)+O\left( \frac{X\log{X}}{YT^{\epsilon_\piup}}\right)+O\left( \frac{X^{\mu_0}}{Y}\right),
\end{equation} 
where $\beta_0$ is the possible exceptional zero of $L(\piup\times\tilde{\piup}, s)$. Let $\tilde{N}_{\piup\times\tilde{\piup}}(\sigma, T)$ be the number of non-exceptional zeroes in the rectangle $1\leq\beta\leq \sigma$ , $|\gamma|\leq T$. From Theorem \ref{thm3} we have 
\begin{equation}
\label{densitytilde}
\tilde{N}_{\piup\times\tilde{\piup}}(\sigma, T)\leq {N}_{\piup\times\tilde{\piup}}(\sigma, T)\leq T^{c(1-\sigma)},
\end{equation}
for a constant $c>0$, uniformly for $0\leq \sigma\leq 1$. Also from Lemma \ref{zerofree} we conclude that there are constants $T_0>0$ and $A_\piup>0$ such that for all $T\geq T_0$ and $\sigma>1-\frac{A_\piup}{\log{T}}$ we have  
$\tilde{N}_{\piup\times\tilde{\piup}}(\sigma, T)=0$. We let $\eta(T)=A_\piup/\log{T}$.
We have
\begin{eqnarray}
\sum_{\substack{{\rho\neq \beta_0}\\{|\gamma|\leq T}}} X^{\beta-1}&=&-\int_{0}^{1^+} X^{\sigma-1} d_\sigma\tilde{N}_{\piup\times\tilde{\piup}}(\sigma, T)\nonumber\\ 
&=&X^{-1}\tilde{N}_{\piup\times\tilde{\piup}}(0, T)+\int_{0}^{1} \tilde{N}_{\piup\times\tilde{\piup}}(\sigma, T) X^{\sigma-1}(\log{X}) d\sigma.\nonumber
\end{eqnarray}
Applying Lemma \ref{LYlemma}(a) and \eqref{densitytilde} gives
$$\sum_{\substack{{\rho\neq \beta_0}\\{|\gamma|\leq T}}} X^{\beta-1}=O(X^{-1}T\log{T})+O\left(\int_{0}^{1-\eta(T)} \left(\frac{T^c}{X} \right)^{1-\sigma}(\log{X}) d\sigma \right).$$
(Here we used the fact that $\tilde{N}_{\piup\times\piup^\prime} (\sigma, T)=0$ when $T\geq T_0$ and $\sigma>1-\eta(T)$.) Setting $T=X^\alpha$ in the above formula yields
$$\sum_{\substack{{\rho\neq \beta_0}\\{|\gamma|\leq T}}} X^{\beta-1}=O(\alpha X^{\alpha-1} \log{X})+O\left(\frac{\exp\left({cA_\piup-\frac{A_\piup}{\alpha}}\right)-X^{\alpha c-1}}{1-\alpha c}\right).$$
From here we see that by choosing $T=X^\alpha$ for sufficiently small $\alpha$  in \eqref{eformula} we will have
\begin{equation*}
\frac{1}{Y}\left( \psi_{\piup\times\tilde{\piup}}(X+Y)-\psi_{\piup\times\tilde{\piup}}(X) \right)\geq \frac{3}{4}-\frac{K_1 X^{1-\epsilon_\piup\alpha} \log{X}}{Y}-\frac{K_2 X^{\mu_0}}{Y}
\end{equation*}
for $X$ sufficiently large (say $X>X_0$), where $K_1$ and $K_2$ are the implied constants in the last two O-terms in \eqref{eformula}. If we choose $Y\geq \max\{8K_1X^{1-\epsilon_\piup\alpha} \log{X}, 8K_2 X^{\mu_0}\}$, then for $X>X_0$ we have
\begin{equation*}
\frac{1}{Y}\left( \psi_{\piup\times\tilde{\piup}}(X+Y)-\psi_{\piup\times\tilde{\piup}}(X) \right)\geq \frac{1}{2}
\end{equation*}
as desired. This completes the proof.
\end{proof}
\begin{proof}
[Proof of Corollary \ref{cor1.7}]
For the first statement, it is enough to note that under the assumption of GRC we have $|a_\piup(n)|\leq m$ and so $$\sum_{X<n\leq X+Y} \Lambda(n) |a_\piup(n)|^2 \ll \sum_{X<n\leq X+Y} \Lambda(n) =\psi(X+Y)-\psi(Y)\ll Y$$
for $Y\geq X^{\theta}$ with $\theta>1/2$ (See \cite[Theorem 6.6]{IK}). The lower bound is a direct consequence of Theorem \ref{thm4} and the bound \eqref{trivial} for local parameters of $\piup\times \tilde{\piup}$ at unramified primes.

For the second assertion without loss of generality assume that $Y\leq X$. A standard computation involving the classical Chebyshev functions yields
\begin{equation}
\label{oneone}
\sum_{X<n\leq X+Y} \Lambda(n) |a_\piup(n)|^2= \sum_{X<p\leq X+Y} (\log{p}) |a_\piup(p)|^2+O\left( {X^{1/2} \log{X}}\right).
\end{equation}
On the other hand by employing  \eqref{trivial} for unramified primes we have have
\begin{equation}
\label{two}
\psi_{\piup\times\tilde{\piup}}(X+Y)-\psi_{\piup\times\tilde{\piup}}(X)=\sum_{X<n\leq X+Y} \Lambda(n) |a_\piup(n)|^2+O\left(X^{1-{2}/{(m^2+1)}}(\log{X})^2\right).
\end{equation}
Substituting \eqref{oneone} in \eqref{two} and then substituting the resulting formula in the left-hand side of \eqref{eformula} and proceeding as in the proof of Theorem \ref{thm4} gives the result.
\end{proof}
\begin{proof}
[Proof of Theorem \ref{thm1.11}]
From Corollary \ref{cor1.7} we know that there exists $0<\nu_\piup<1$ such that for $X>X_0$ and $Y\geq X^\theta$ with $\nu_\piup<\theta<1$, we have
$$Y\ll \sum_{X<p\leq X+Y} (\log{p}) |a_\piup(p)|^2 \ll (\log{X}) (\#\{p;~X<p\leq X+Y~{\rm and}~a_\piup(p)\neq 0\}).$$
This proves the first assertion. Setting $X=p$ and $Y=p^\theta$ in the above inequality implies that for large $p$ we have $\#\{{\rm prime}~q;~p<q\leq p+p^\theta~{\rm and}~a_\piup(q)\neq 0\}\neq 0$. This shows that $j_\piup(p) \ll p^\theta$.
\end{proof}
Finally the results for newforms and the Ramanujan $\tau$-function follow from Theorem \ref{thm1.11} since these $L$- functions satisfy GRC.
\medskip\par
\noindent{\bf Acknowledgement}  We thank the referee for many helpful comments and suggestions. We are also grateful to Satadal Ganguly for his comments on an earlier draft of this paper. Also the first author would like to thank Adam Felix and Kumar Murty for the useful discussion regarding this work.

\end{document}